\documentclass[11pt, reqno]{amsart}

\makeatletter\@addtoreset{equation}{section}\makeatother

\usepackage{amsmath, amsthm, epsfig, latexsym, amssymb, bbm}
\usepackage[dvips]{color}

\newcommand{\eps}{\varepsilon}

\newcommand{\R}{\mathbb R}

\newcommand{\cF}{F}
\newcommand{\cG}{\mathcal G}

\newcommand{\E}{\mathbb E}
\renewcommand{\P}{\mathbb P}

\newcommand{\N}{\mathbb N}

\renewcommand{\P}{\mathbb P}

\newtheorem{rem}{Remark}
\newtheorem{theorem}{Theorem}

\newtheorem{lemma}[theorem]{Lemma}

\theoremstyle{remark}

\newcommand{\sfrac}[2] {\mbox{$\frac{#1}{#2}$}}

\def\1{{\mathchoice {1\mskip-4mu\mathrm l}      
{1\mskip-4mu\mathrm l}
{1\mskip-4.5mu\mathrm l} {1\mskip-5mu\mathrm l}}}

\begin{document}


\title[Condensation effect in Kingman's model]
{Emergence of condensation in Kingman's model\\ of selection and mutation \vspace{0.2cm}
\centerline{\rm \small \today}}

\author[Steffen Dereich and Peter M\"orters ]{Steffen Dereich and Peter M\"orters}

\vspace{0.2cm}

\begin{abstract}
We describe the onset of condensation in the simple model for the balance between selection
and mutation given by Kingman in terms of a scaling limit theorem. Loosely speaking, this shows 
that the wave moving towards genes of maximal fitness has the shape of a gamma distribution. We 
conjecture that this wave shape is a universal phenomenon that can also be found in a variety 
of more complex models, well beyond the genetics context, and provide some further evidence for this.
\end{abstract}
\bigskip

\vspace{0.2cm}

\maketitle

\thispagestyle{empty}

\section{Introduction and statement of the result}

In~\cite{K78} Kingman proposes and analyses a simple model for the distribution of fitness in a population undergoing
selection and mutation. The characterisitic feature of this model is that the fitness of genes before and after mutation 
is modelled as independent, the mutation having destroyed the biochemical `house of cards'
built up by evolution. Kingman shows that in his model the distribution of the fitness in the population 
converges to a limiting distribution. There are two phases: When \emph{mutation} is favoured over selection, the limiting 
distribution is a skewed version of the fitness distribution of a mutant. But if \emph{selection} is favoured over mutation,
a condensation effect occurs, and we find that a positive proportion of the population in late generations has fitness very 
near the optimal value, leading to the emergence of an atom at the maximal fitness value in the limiting distribution.  
Physicists have argued that this is akin to the effect of Bose-Einstein condensation, in which for a dilute gas of weakly 
interacting bosons at very low temperatures a fraction of the bosons occupy the lowest possible quantum state, 
see for example~\cite{BFF09}. In the present paper, we focus on the Kingman model 
and discuss the form of the fitness distribution for that part of the population that eventually form the atom in the limiting 
distribution. After stating our theorem and giving a proof we will draw comparisons to other models in a discussion section at 
the end of this paper.
\medskip

Mathematically, Kingman's model consists of a sequence of probability measures $(p_n)$ on 
the unit interval $[0,1]$ describing the distribution of fitness values in the $n$th generation 
of a population. The parameters of the model are a   mutant fitness distribution 
$q$ on $[0,1]$ and some $0<\beta<1$ determining the relation between mutation and selection. If $p_n$ 
is the fitness distribution in the $n$th generation we denote
by $$w_n=\int x \, p_n(dx)$$ the mean fitness and define
$$p_{n+1}(dx)= (1-\beta) \, w_n^{-1} x \, p_n(dx) + \beta \, q(dx).$$
Loosely speaking, a proportion $1-\beta$ of the genes in the new generation are resampled from the
existing population using their fitness as a selective criterion, and the rest have undergone mutation and
are therefore sampled from the fitness distribution~$q$.
\medskip

We assume throughout that the mutant fitness distribution near its tip is stochastically larger
than the fitness distribution in the inital population, in the sense that the moments
$$m_n:=\int x^n \, p_0(dx) \quad \mbox{ and } \mu_n:=\int x^n \, q(dx)$$
satisfy $$\lim_{n\to\infty} \frac{m_n}{\mu_n}=0.$$
Under this (or, indeed, a weaker) assumption, Kingman showed 
that $(p_n)$ converges to a limit distribution $p(dx)$, which does not depend on~$p_0$. Moreover, $p$ is absolutely continuous 
with respect to $q$ if and only if $$\beta \int_0^1 \frac{q(dx)}{1-x} \geq 1.$$
Otherwise, 
\begin{equation}\label{gam}
\gamma(\beta):=1-\beta \int_0^1 \frac{q(dx)}{1-x}>0,
\end{equation}
and this is the case of interest to us. In this case the limiting distribution $p(dx)$ still exists, but 
it has an atom at the optimal fitness~$1$,  an effect called \emph{condensation}. The limiting distribution 
does not depend on~$p_0$ and equals
$$p(dx) = \beta \frac{q(dx)}{1-x} +  \gamma(\beta) \, \delta_1 (dx).$$
Our main result describes the dynamics of condensation in terms of a scaling limit
theorem which zooms into the neighbourhood of the maximal fitness value and
shows the shape of the `wave' eventually forming the condensate, see~Figure~1.

\begin{figure}[h]
  \centerline{ \hbox{ \psfig{file=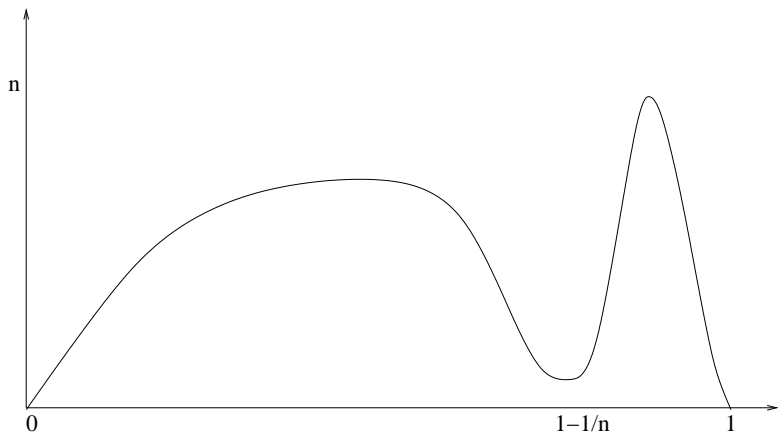,height=2.5in}}}
  \caption{\footnotesize Schematic picture of $p_n$. On the right the \emph{wave} is a high peak with length of order $1/n$ and height of order $n$. 
  By contrast, the \emph{bulk} has height and length of order one. }
\end{figure}

\pagebreak[3]

\begin{theorem}
Suppose that the fitness distribution $q$ satisfies
\begin{equation}\label{tailass}
\lim_{h\downarrow 0} \frac{q(1-h,1)}{h^\alpha}=1,
\end{equation} 
where $\alpha>1$,
and that~\eqref{gam} holds. Then, for $x>0$, 
\begin{equation}\label{wave}
\lim_{n\uparrow \infty} p_n(1-\frac{x}{n},1) = \frac{\gamma(\beta)}{\Gamma(\alpha)}\int_0^x y^{\alpha-1} e^{-y} \, dy.
\end{equation} 
\end{theorem}
\medskip

We remark that the total mass in the `wave' moving towards 
the maximal fitness value agrees with the mass of the atom in the limiting distribution~$p(dx)$.
Its rescaled shape is that of a \emph{gamma distribution} with shape parameter~$\alpha$.

\section{Proof of Theorem~1}

Note that
$$\mu_n = \int x^n \, q(dx) \sim \Gamma(\alpha+1)\, n^{-\alpha},$$
where the asymptotics is easily derived from~\eqref{tailass}, and
note that
\begin{align}\label{eq1304-1}\sum_{n=0}^\infty \mu_n = \int_0^1 \frac {q(dx)}{1-x}=\frac1{\beta} \,  (1-\gamma(\beta)).
\end{align}
Also define
$$W_n:=w_1\cdots w_n.$$
Given the family $(W_n)_{n\geq 1}$ the fitness distributions can be obtained as
\begin{equation}\label{sol}
p_n(dx)=\sum_{r=0}^{n-1} \frac{W_{n-r}}{W_n} \, (1-\beta)^r \beta \, x^r \, q(dx) 
+ \frac{1}{W_n}\,(1-\beta)^n \, x^n \, p_0(dx),
\end{equation}
see \cite[(2.1)]{K78}. 
The main tool in the proof is therefore the following lemma.

\begin{lemma}\label{thelemma}
We have, as $n\uparrow \infty$, 
$$W_n \sim c\,  n^{-\alpha} (1-\beta)^{n-1},$$
where $$c=\frac{\beta}{\gamma(\beta)}\, \Gamma(\alpha+1)\,  \sum_{k=1}^\infty W_k \, (1-\beta)^{1-k}.$$
\end{lemma}

\begin{proof}
Integrating~\eqref{sol} we obtain~\cite[(2.3)]{K78}
$$W_n=\sum_{r=1}^{n-1} W_{n-r} \, (1-\beta)^{r-1} \beta \, \mu_r 
+ (1-\beta)^{n-1} \, m_n.$$
Abbreviate $u_n:= W_n \, (1-\beta)^{1-n}.$
Then $u_n$ satisfies the \emph{renewal equation}
$$u_n = \sfrac{\beta}{1-\beta}\, \sum_{r=1}^{n-1}  u_{n-r}  \mu_r + m_n, \qquad \mbox{ for } n\geq 1.$$
Using (\ref{eq1304-1}), we obtain $\frac{\beta}{1-\beta}\sum_{n=1}^\infty \mu_n=1-\frac{\gamma(\beta)}{1-\beta}<1$. Hence, the renewal theorem, see e.g.~\cite[XXXIII.10, Theorem~1]{F}, implies that
$$\sum_{n=1}^\infty u_n = \frac{\sum_{n=1}^\infty m_n}{1- \frac{\beta}{1-\beta}\, \sum_{n=1}^\infty \mu_n}=\frac{1-\beta}{\gamma(\beta)} \sum_{n=1}^\infty m_n<\infty$$
where the finiteness follows since $m_n$ is bounded by a constant multiple of $\mu_n$ and 
$$
\sum_{n=0}^\infty \mu_n  =\int \frac {q(dx)}{1-x}  <\infty.
$$ 
Fix $\delta>0$ and $0<\eps<\eta<1$ and suppose
$n$ is large enough such that $\eta n\leq n-1$ and
$$\mu_r\leq (\Gamma(\alpha+1)+\delta)\, r^{-\alpha} \quad \mbox{ for all } r\geq (1-\eta) n.$$ 
For an inductive argument suppose that $c_1, \ldots, c_r$ are chosen such that $u_r\leq c_r\, r^{-\alpha}$ for all 
$\eps n \leq r \leq n-1$.  Then one has for $r=1,\dots,n-1$
$$
u_r \mu_{n-r} \leq \begin{cases} (1-\eps)^{-\alpha} (\Gamma(\alpha+1)+\delta) n^{-\alpha} u_r &\text{ if }r\leq \eps n,\\
c_r\,(\Gamma(\alpha+1)+\delta) r^{-\alpha} (n-r)^{-\alpha}  &\text{ if } \eps n\leq r \leq \eta n,\\
c_{r}\, \eta^{-\alpha} n^{-\alpha} \mu_{n-r} &\text{ if }  \eta n\leq r,
\end{cases}
$$
so that
\begin{align}
u_n 
& \leq  (1-\eps)^{-\alpha}  \,  \sfrac{\beta}{1-\beta}\, (\Gamma(\alpha+1)+\delta)\, \Bigl(\sum_{r=1}^{\infty}  u_r\Bigr) \,  n^{-\alpha}\,  \notag\\
& \phantom{space} +  \sfrac{\beta}{1-\beta}\, (\Gamma(\alpha+1)+\delta)\,  \Bigl(\sfrac1n \sum_{r=\lfloor \eps n\rfloor +1}^{\lfloor \eta n\rfloor}  
c_r\, \bigl(\sfrac rn\bigr)^{-\alpha}\bigl(1-\sfrac rn\bigr)^{-\alpha} \Bigr) n^{1-2\alpha}\, \label{defcn}\\
& \phantom{space} +  \sfrac{\beta}{1-\beta}\,\eta^{-\alpha} \, \Bigl(\sum_{r=\lfloor \eta n\rfloor+1}^{n-1}  c_r \, \mu_{n-r}\Bigr)  n^{-\alpha}+ m_n
=: c_n n^{-\alpha}.\notag
\end{align}
By induction this yields a sequence $(c_n)$ with
$u_n \leq c_n\, n^{-\alpha}$ for all~$n\geq 1$.
\smallskip

Using that $m_n n^{\alpha} \to 0$ by assumption, 
and that the term~\eqref{defcn} is bounded by a constant multiple~of
$$n^{1-2\alpha}\, \int_{\eps }^{\eta }  dr\, r^{-\alpha}(1-r)^{-\alpha}  \ll n^{-\alpha},$$
we see that $(c_n)$ converges to the unique solution~$c^*=c^*(\eps, \delta, \eta)$ of
$$c^*=(1-\eps)^{-\alpha}  \, \sfrac{\beta}{1-\beta}\, (\Gamma(\alpha+1)+\delta)\, \sum_{r=1}^{\infty}  u_r 
+  c^*\,\eta^{-\alpha} \, \sfrac{\beta}{1-\beta}\,\sum_{r=1}^\infty \mu_r.$$
Recalling that $\beta \sum_{r=1}^\infty \mu_r = 1-\gamma(\beta)-\beta$, and
letting $\eps, \delta\downarrow 0$ and $\eta\uparrow 1$ we see that $c^*(\eps, \delta, \eta)$ converges to
$$c=\frac{\beta}{\gamma(\beta)}\, \Gamma(\alpha+1) \, \sum_{k=1}^\infty u_k= 
\frac{\beta}{\gamma(\beta)}\, \Gamma(\alpha+1)\,  \sum_{k=1}^\infty W_k \, (1-\beta)^{1-k},$$
which yields the upper bound.
%
The lower bound can be derived similarly.
\end{proof}

To complete the proof using the lemma, we look at~\eqref{sol} and get
\begin{align*}
p_n\big(1- \frac{x}{n},1 \big) & =
\sum_{r=0}^{n-1} \frac{W_{n-r}}{W_n} \, (1-\beta)^r \beta \, \int_{1-x/n}^1 y^r \, q(dy) 
+ \frac{1}{W_n}\,(1-\beta)^n \, \int_{1-x/n}^1 y^n \, p_0(dy). \\
\end{align*}
The second term vanishes asymptotically, as
$$\frac{1}{W_n}\,(1-\beta)^n \, \int_{1-x/n}^1 y^n \, p_0(dy)
\sim (1-\beta)\, \frac{m_n}{c n^{-\alpha}} \frac{\int_{1-x/n}^1 y^n \, p_0(dy)}{\int_0^1 y^n \, p_0(dy)}\to 0,$$
using our assumption that $m_n/\mu_n\to 0$.  The first term is  asymptotically equivalent to 
$$n^\alpha \sum_{r=0}^{n-1} W_{n-r} \,  c^{-1}\,(1-\beta)^{1-n+r}\beta\, \int_{1-x/n}^1 y^{r} \, q(dy).$$
By chosing a large~$M$, the contribution coming from terms with $r\leq n-Mn^{1/\alpha}$ can be bounded by
a constant multiple of
$$(n-Mn^{1/\alpha}) \, \Big(\frac{n}{Mn^{1/\alpha}}\Big)^\alpha q\big(1-\sfrac{x}{n},1\big),$$
which is bounded by an arbitraily small constant. For the remaining terms we can now use that
$$\begin{aligned}
\int_{1-x/n}^1 y^{sn} \, q(dy) & 
\sim   \int_x^0 e^{-as} \, dq(1-\sfrac{a}n,1)
\sim   \alpha\, n^{-\alpha} \int_0^x a^{\alpha-1} e^{-as} \, da,
\end{aligned}$$
and a change of variables to obtain equivalence to
$$\alpha\,\beta\,c^{-1} \Big(\sum_{m=1}^{\infty} W_{m} \,(1-\beta)^{1-m}\Big) \, \int_{0}^x a^{\alpha-1} e^{-a}\, da,$$
and the result follows as, by Lemma~\ref{thelemma},
$$\alpha\,\beta\,c^{-1} \Big(\sum_{m=1}^{\infty} W_{m} \,(1-\beta)^{1-m}\Big) =  \frac{\gamma(\beta)}{\Gamma(\alpha)},$$
as required.

\section{Discussion}

Kingman's model is on the one hand one of the simplest models in which a condensation effect can
be observed, on the other hand it is sufficiently rich to  study the emergence of condensation as 
a dynamical phenomenon. The simplicity of the model allows a rigorous treatment with elementary means,
but we believe that our calculation has far reaching consequences as a variety of much more complex 
models in quite diverse areas of science have similar features. Among the models we expect to share
many features with Kingman's model are models of the physical phenomenon of Bose-Einstein condensation, 
of wealth condensation in macroeconomics, or the emergence of traffic jams.
\medskip

\pagebreak[3]

Our \emph{main conjecture} is that in a large universality class of models in which effects similar to 
mutation and selection compete effectively on a bounded and continuous statespace, the `wave' moving 
towards  the maximal state forming the condensate is of a Gamma shape. 
\medskip

Random models which are suitable test cases for our universality claim arise, for example, in
the study of random permutations with cycle weights. Here the probability of a permutation~$\sigma$
in the symmetric group on $n$~elements is defined as
$$\P_n(\sigma)= \frac1{n!h_n} \prod_{j\geq 1} \theta_j^{R_j(\sigma)},$$
where $R_j(\sigma)$ is the number of cycles of length~$j$ in~$\sigma$ and $h_n$ is a
normalisation constant. For our investigation we focus on the case that $\theta_j\sim j^{\gamma}$
for $\gamma\in\R$. We now discuss results of Betz, Ueltschi and Velenik~\cite{BUV11} and Ercolani and 
Ueltschi~\cite{EU11} in our context.
\medskip

Our interest is in the \emph{empirical cycle length distribution} which is the random measure on $[0,1]$
given by
$$\mu_n= \frac1n \sum_{i=1}^n \lambda_i\, \delta_{\frac{\lambda_i}{n}},$$ 
where the integers $\lambda_1\geq \lambda_2 \geq \cdots$ are the ordered cycle lengths of a permutation
chosen randomly according to $\P_n$. 
The asymptotic behaviour of $\mu_n$ shows
three phases depending on the value of the parameter~$\gamma$, see Table~1 in~\cite{EU11}:
\begin{itemize}
\item If $\gamma<0$ large  cycles are preferred and the empirical cycle length distribution concentrates 
asymptotically in the point $1$, 
\item if $\gamma=0$ there is no condensation and we have convergence to a beta distribution, 
\item if $\gamma>0$ we see a preference for short cycles and  the empirical cycle length distribution concentrates 
asymptotically in the point $0$.  
\end{itemize}
In the two phases in which see a condensation effect we have partial information on the shape of the wave,
which is consistent with our universality claim.
\medskip

Let us first look at the case $\gamma>0$ when the empirical cycle length distribution concentrates in
the left endpoint of our domain, i.e. the normalised cycle lengths vanish asymptotically.
In this case Theorem~5.1 of~\cite{EU11} shows that, for $\alpha=\frac{\gamma}{\gamma+1}$,
$$\lim_{n\to\infty} \E\big[\mu_n[0,\sfrac{x}{n^\alpha})\big]=\frac1{\Gamma(\gamma+1)}\,  \int_0^x y^\gamma e^{-y}\, dy,$$
i.e.\ focusing on the left edge of the domain in the scale $1/n^\alpha$ we see a gamma distributed wave shape with
parameter~$\gamma$, at least in 
the mean.  It is a natural conjecture that this convergence holds not only  in expectation, but also in probability, and  establishing 
this fact is subject of an  ongoing project.
\medskip

If  $\gamma<0$ large cycles are preferred. Here the situation is slightly different because the wave sweeping
towards the maximal normalised cyclelength is on the critical scale $1/n$ and this means that we expect that the discrete
nature of $\mu_n$ is retained in the limit. 
\medskip

More precisely,  Theorem~3.2 of~\cite{BUV11} implies that
$$\lim_{n\to\infty} \E\big[\mu_n[1-\sfrac{m}{n},1]\big]= \sfrac12\, \sum_{n=0}^m e^{-c^*n}h_n,$$
where $c^*$ is a `Malthusian parameter' chosen such that
$$\sum_{n=1}^\infty e^{-c^*n}h_n=1.$$
We further note that $h_n\sim C\,n^{\gamma-1}$ by \cite[(7.1)]{EU11} and so we are still able to recognise 
a discrete form of a gamma distribution with parameter $\gamma$ in this case.
\medskip

The most elaborate model in which we were able to test our hypothesis is a random network model with fitness. We now
give an informal preview of forthcoming results of Dereich~\cite{D12}, which are motivated by a problem of Borgs et al.~ \cite{BCDR07}.
\medskip

A preferential attachment network model is a sequence of random  graphs $(\cG(n))_{n\in\N}$ that is built dynamically: one starts with a graph $\cG(1)$ consisting of a single vertex $1$ and, in general, the graph  $\cG({n+1})$ is built by adding the vertex $n+1$ to the graph $\cG(n)$ and by insertion of edges connecting the new vertex to the graph $\cG(n)$ according to an attachment rule. Typically, the attachment rule rewards vertices that already have a high degre: in most cases the degree of a vertex has an affine influence on its attractiveness in the collection of new edges. In a preferential attachment model with fitness one additionally assigns each vertex an intrinsic fitness, a positive number, which has a linear impact on its attractiveness in the network formation.\smallskip

Let us be more precise about the variant of the network model to be considered in the rest of this paper. 
We consider a sequence of random \emph{directed} graphs $(\cG(n))_{n\in\N}$ and denote by
$$
\mathrm{imp}_{n}(m):= \mathrm{indegree}_{\cG(n)}(m) +1
$$
the \emph{impact} of the vertex $m\in\{1,\dots,n\}$ in $\cG(n)$. Further, let $F_1,F_2,\dots$ denote a sequence of independent $q$-distributed random variables modeling the fitness of the individual vertices $1,2,\ldots$. The attachment rule is as follows: given the graph $\cG(n)$ and all fitnesses, link $n+1$  to  each individual vertex $m\in\{1,\dots,n\}$ with  an independent Poisson distributed number of edges with parameter
$$
\frac 1{n\,Z_n}\, {F_m\, \mathrm{imp}_{n}}(m),
$$
where $Z_n$ is a normalisation which depends only on $\cG(n)$ and the fitnesses. Note that all links point from new to old vertices so that orientations can be recovered from the undirected set of edges.
We consider two types of normalisations:

\begin{enumerate}
\item \emph{adaptive normalisation}: $Z_n=\frac 1{\lambda n} \sum_{m=1}^n \cF_m \,\mathrm{imp}_{n}(m)$ for a parameter $\lambda>0$,\\[-2mm]
\item \emph{deterministic normalisation}: $(Z_n)$ is a deterministic sequence. 
\end{enumerate}

In the case of adaptive normalisation,  the outdegree of $n+1$ is Poisson distributed with parameter~$\lambda$, even when conditioning on the graph $\cG(n)$. Hence, the total  number of edges is almost surely of order $\lambda n$ so that $\frac 1{n} \sum_{m=1}^n \mathrm{imp}_{n} (m)$ converges almost surely to $\lambda+1$.\smallskip

The analogue of $p_n$ is the \emph{impact measure} given by  
$$\Xi_n= \frac 1{n} \sum_{m=1}^n \mathrm{imp}_{n} (m)\, \delta_{\cF_m}.
$$
It measures the contribution of the vertices of a particular fitness to the total impact. 
\smallskip

As observed in \cite{BiBa01} and verified for a  different variant of the model in \cite{BCDR07}, network models with fitness show a phase transition similar to Bose-Einstein condensation. The verification of this phase transition in the variant considered here is conducted in \cite{DO12}.
\smallskip 

For \emph{adaptive} normalisation  two regimes can be observed
 \begin{enumerate}
\item[{[}FGR{]}] $ \int \frac 1{1-x} \, q(d x)\geq 1+\lambda$: the \emph{fit-get-richer phase},\\[-2mm]
\item[{[}BE{]}]  $\int \frac 1{1-x} \,  q(d x)<1+\lambda$: the \emph{Bose-Einstein phase}  or \emph{innovation-pays-off phase}.
\end{enumerate}

In the fit-get-richer phase, the random measures  $(\Xi_n)_{n\in\N}$ converge almost surely in the weak topology to the measure $\Xi$ on $(0,1]$ given by
 $$\Xi(d x)=  \frac{ \lambda^*}{\lambda^*-x}\, q(d x),$$
where  $\lambda^*\in[1,\infty)$ denotes the unique solution to  
$$ \int \frac {\lambda^*}{\lambda^*-x} \, q(d x)= 1+\lambda,$$ 
whereas, in the Bose-Einstein phase, one observes convergence to
$$\Xi(d x)= \frac 1{1-x} \, q(d x) + \Bigl(1+\lambda-\int \frac 1{1-y} \, q(d y)\Bigr) \delta_1(d x).$$
\smallskip

In order to analyse the emergence of the condensation phenomenon, we consider the preferential attachment model with \emph{deterministic} normalisation. 
We assume that $q$ is regularly varying at~$1$ with representation 
$$
q(1-h,1)= h^\alpha \,\ell(h),
$$
where $\ell:[0,1]\to(0, \infty)$ is a slowly varying function. In order to replicate the Bose-Einstein phenomenon in the model with deterministic normalisation, one needs to choose $(Z_n)$ appropriately. For $1\leq m\leq n$, let
$$
\Upsilon[ m,n] := \sum_{k=\lfloor  m\rfloor}^{\lfloor n\rfloor} \frac {1-Z_k}k.
$$
The Bose-Einstein phenomenon can be replicated by choosing $(Z_n)$ such that
$$
1- Z_n\sim \alpha  (\log n)^{-1} 
$$
and such that the limit
\begin{align}\label{eq1904-1}
\gamma:=\lim_{n\to\infty} \frac {\alpha}{\alpha-1} \Gamma(\alpha) \,\frac{(\log n)^\alpha \cdot \log (\log n)^\alpha}{\ell((\log n)^{-1})} \,\exp\{ \Upsilon[\log n, n]\} 
\end{align}
exists. We stress that such a normalisation can be found  for various fitness distributions $q$ and we  refer the reader to the article~\cite{D12} for the details.

\begin{theorem}Under the above assumptions, one has, for  $x>0$,
$$
\lim_{n\to\infty} \Xi_n\Bigl(1-\frac x{\log n},1\Bigr) = \frac{\gamma}{\Gamma(\alpha)}\int_0^x y^{\alpha-1} e^{-y} \, dy, \text{ in probability}.
$$
For any measurable set  $A\subset [0,1]$ with $1\not\in \partial A$,
one has
$$
\lim_{n\to\infty} \Xi_n(A) = \Xi(A), \text{ in probability},
$$
for  the measure $\Xi$ on $[0,1]$ given by
$$
\Xi(dx) = \frac 1{1-x} \,q(dx) + \gamma \,\delta_1(dx).
$$
\end{theorem}

\begin{rem} {\rm In most cases one cannot give an explicit representation for a normalisation $(Z_n)$ satisfying~(\ref{eq1904-1}).  On  first sight, this might be 
suprising since the $(Z_n)$ play a r\^ole analogous to $(W_n)$ in the Kingman model where the analysis is feasible. The difference of both models comes 
from the stochastic nature of the network model. In order to analyse the network model one could start to work with expectations resulting in a mean field model 
similar to the Kingman model. However, the expectations for $\Xi_n$ are dominated by configurations that are not seen in typical realisations:  vertices  of 
particular high fitness that are born very early contribute most although  being not present typically.
To compensate this the normalisations in the network model have to be slightly smaller than  a mean field model would suggest. Vertices of particularly high 
fitness have an impact only with a delay. This causes the $\Upsilon[\log n,n]$ term in (\ref{eq1904-1}) and makes explicit representations for $(Z_n)$ in 
many cases unfeasible.}
\end{rem}
\smallskip

We conclude our discussion with the remark that the case of \emph{unbounded fitness distribution} is also of considerable interest. In this
case Park and Krug~\cite{PK08} have studied the analogue of Kingman's model and (in a particular case) observed emergence of a travelling
wave of Gaussian shape. They also conjecture that this behaviour is of universal nature. 
\bigskip

\ \\[-1mm]

{\bf Acknowledgments:} The second author acknowledges useful discussions with Daniel  Ueltschi at the Oberwolfach workshop
\emph{Interplay of analysis and probability in physics}, January 2012. We would like to thank Marcel Ortgiese for agreeing
to include a preview of~\cite{DO12} in our discussion.

\end{document}